\documentclass[12pt,amsfonts]{article}
\usepackage[margin=1in]{geometry}  
\usepackage{graphicx}              
\usepackage[latin1]{inputenc}
\usepackage{amsmath, amssymb, amsthm, epsfig}               
\usepackage{enumerate}
\usepackage{amsfonts}              
\usepackage{amsthm}                
\usepackage{amsthm}
\usepackage{enumerate}
\theoremstyle{plain}
\usepackage{color}
\def\dis{\displaystyle}
\def\nd{\noindent}
\def\thend{\rule{3mm}{3mm}}

\newtheorem{claim}{Claim}[section]
\newtheorem{theorem}{Theorem}[section]

\newtheorem{proposition}{Proposition}[section]
\newtheorem{lemma}{Lemma}[section]

\newtheorem{definition}{Definition}[section]
\newtheorem{corollary}{Corollary}[section]
\newtheorem*{theorem*}{Theorem}

\numberwithin{equation}{section}

\begin{document}
\title{Multiple positive solutions for a Schr\"{o}dinger logarithmic equation}
\author{ Claudianor O. Alves\footnote{C.O. Alves was partially supported by CNPq/Brazil  304804/2017-7.} \,\, and \,\, Chao Ji \footnote{C. Ji was partially supported by Shanghai Natural Science Foundation(18ZR1409100).}}

\maketitle

\begin{abstract}
This article concerns  with the existence of multiple positive solutions for the following logarithmic Schr\"{o}dinger equation
$$
\left\{
\begin{array}{lc}
-{\epsilon}^2\Delta u+ V(x)u=u \log u^2, & \mbox{in} \quad \mathbb{R}^{N}, \\
u \in H^1(\mathbb{R}^{N}), & \;  \\
\end{array}
\right.
$$
where  $\epsilon >0$,  $N \geq 1$  and $V$ is a continuous function with a global minimum. Using variational method, we prove that for small enough $\epsilon>0$,  the "shape" of the graph of the function $V$ affects  the number of nontrivial solutions.
\end{abstract}



\section{Introduction}

Recently, the logarithmic Schr\"odinger equation given by
$$
i \epsilon \partial_t \Psi = - \epsilon^{2}\Delta \Psi + (W(x	)+w)\Psi- \Psi \log|\Psi|^2, \ \Psi:[0, \infty) \times \mathbb{R}^{N} \rightarrow \mathbb{C}, \ N \geq 1,
$$
 has  received a special attention because it appears in a lot of physical applications, such as quantum mechanics, quantum optics, nuclear physics, transport and diffusion phenomena, open quantum systems, effective quantum gravity, theory of superfluidity and Bose-Einstein condensation (see \cite{z} and the references therein). In its turn, standing waves solution, $\Psi$, for this logarithmic Schr\"odinger equation  is related to solutions of the equation
$$
-\epsilon^{2}\Delta u+ V(x)u=u \log u^2,  \quad \mbox{in} \quad \mathbb{R}^{N}.
$$

Besides the importance in applications, this last  equation is very interesting in the mathematical point of view, because it arises a lot of difficulties to apply variational methods in order to find a solution for it. The natural candidate for the associated energy functional would formally be the functional
$$
\widehat{I}_\epsilon(u)=\frac{1}{2}\displaystyle \int_{\mathbb{R}^N}(\epsilon^{2}|\nabla u|^2+V(x)|u|^2)dx-\displaystyle \int_{\mathbb{R}^N} H(u)dx,
$$
\noindent where
\begin{equation}\label{lula}
H(t)=\displaystyle\int^t_0s\log s^2ds=\displaystyle \frac{-t^2}{2}+\displaystyle \frac{t^2\log t^2}{2}, \quad \forall t \in \mathbb{R}
\end{equation}
that is,
$$
\widehat{I}_\epsilon(u)=\frac{1}{2}\displaystyle \int_{\mathbb{R}^N}(\epsilon^{2}|\nabla u|^2+(V(x)+1)|u|^2)dx-\displaystyle \frac{1}{2}\int_{\mathbb{R}^N} u^2 \log u^2\,dx.
$$
However, this functional is not well defined in $H^{1}(\mathbb{R}^N)$ because there is $u \in H^{1}(\mathbb{R}^N)$ such that $\int_{\mathbb{R}^N}u^{2}\log u^2 \, dx=-\infty$. In order to overcome this technical difficulty some authors have used different techniques.

In \cite{squassina},  d'Avenia, Montefusco and Squassina have studied the existence of multiple solutions for a logarithmic elliptic equation of the type
$$
\left\{
\begin{array}{lc}
-\Delta u+ u=u \log u^2, & \mbox{in} \quad \mathbb{R}^{N}, \\
u \in H^1(\mathbb{R}^{N}). & \;  \\
\end{array}
\right.
\eqno{(P_1)}
$$
The authors obtained solutions for this equation by applying the non-smooth critical point theory, found in Degiovanni and Zani \cite{DZ}, to the energy functional defined on the space of radial functions $H_{rad}^{1}(\mathbb{R}^N)$. In \cite{ASZ}, d'Avenia, Squassina and Zenari have used the same approach to show the existence of solution for a  fractional logarithmic Schr\"odinger equation of the type
$$
\left\{
\begin{array}{lc}
(-\Delta)^{s} u+ u=u \log u^2, & \mbox{in} \quad \mathbb{R}^{N}, \\
u \in H^s(\mathbb{R}^{N}), & \;  \\
\end{array}
\right.
\eqno{(P_2)}
$$
for $s \in (0,1)$ and $N>2s$.

In \cite{sz}, Squassina and Szulkin have showed the existence of multiple solutions for the following class of problem
$$
\left\{
\begin{array}{lc}
-\Delta u+ V(x)u=Q(x)u \log u^2, & \mbox{in} \quad \mathbb{R}^{N}, \\
u \in H^1(\mathbb{R}^{N}), & \;  \\
\end{array}
\right.
\eqno{(P_3)}
$$
where $V,Q:\mathbb{R}^N \to \mathbb{R}$ are 1-periodic continuous functions verifying
$$
\min_{x \in \mathbb{R}^N}Q(x)>0 \quad \mbox{and} \quad \min_{x \in \mathbb{R}^N}(V+Q)(x)>0.
$$
In that paper, the authors have used the minimax principles for lower semicontinuous functionals, developed by Szulkin \cite{szulkin}, to prove the existence of geometrically distinct multiple solutions and the existence of a ground state solution. The multiple solutions follow by genus theory found in \cite{szulkin}, while the existence of ground state follows by a specific deformation lemma, see \cite[Lemma 2.14]{sz}.

Later, Ji and Szulkin in \cite{cs} have established the existence of multiple solutions for a problem of the type
$$\label{encontrar}
\left\{
\begin{array}{lc}
-\Delta u+ V(x)u=u \log u^2, & \mbox{in} \quad \mathbb{R}^{N}, \\
u \in H^1(\mathbb{R}^{N}), & \;  \\
\end{array}
\right.
\eqno{(P_4)}
$$
where $V:\mathbb{R}^N \to \mathbb{R}$ is a continuous function that satisfies
$$
\lim_{|x|\to +\infty}V(x)=V_\infty \quad \mbox{where} \quad V_\infty +1 \in (0,+\infty].
$$

With the same approach explored in \cite{sz}, the above authors showed that if $V_\infty=+\infty$, then $(P_4)$ has infinitely many solutions. If $V_\infty \in (-1,+\infty)$ and the spectrum $\sigma(-\Delta+V+1) \subset (0,+\infty)$, then problem $(P_4)$ has a ground state solution. In \cite{TZ}, Tanaka and Zhang have studied the existence of solution for $(P_3)$. In that very nice paper the authors have observed that the positivity of $V$ is not essential.

Finally, in a recent work,  Alves and de Morais Filho in \cite{AlvesdeMorais} have used the minimax method found in \cite{szulkin} to show the existence and concentration of positive solution for the problem
\begin{equation}\label{1}
\left\{
\begin{array}{lc}
-{\epsilon}^2\Delta u+ V(x)u=u \log u^2, & \mbox{in} \quad \mathbb{R}^{N}, \\
u \in H^1(\mathbb{R}^{N}), & \;  \\
\end{array}
\right.
\end{equation}
where  $\epsilon >0, N \geq 1$  and $V$ is a continuous function with a global minimum.

Motivated by results found  in the above mentioned papers, in the present paper we intend to study the existence of multiple solutions for problem (\ref{1})
by supposing the following conditions on potential $V$:
\begin{itemize}
	\item[\rm ($V1$)] $V:\mathbb{R}^N\rightarrow\mathbb{R}$ is a continuous function such that
	$$
	\lim_{|x|\rightarrow\infty}V(x)=V_\infty,
	$$
	with $0<V(x)<V_\infty$ for any $x\in \mathbb{R}^N$.

	\item[\rm ($V2$)]There exist $l$ points $z_1,z_2,...,z_l$ in $\mathbb{R}^N$ with $z_1=0$ such that
	$$
	1=V(z_i)=\min_{x\in\mathbb{R}^N}V(x),~~~{\rm for}~~1\leq i\leq l.$$
	
\end{itemize}

By a change of variable, we know that problem (\ref{1}) is equivalent to the problem
\begin{equation}\label{1'}
\left\{
\begin{array}{lc}
-\Delta u+ V(\epsilon x)u=u \log u^2, & \mbox{in} \quad \mathbb{R}^{N}, \\
u \in H^1(\mathbb{R}^{N}). & \;  \\
\end{array}
\right.
\end{equation}

\begin{definition}
For us, a positive solution of (\ref{1'}) means a positive function $u \in H^1(\mathbb{R}^{N}) \setminus \{0\}$ such that $ u^2\log u^2 \in L^1(\mathbb{R}^{N})$ and
\end{definition}

\begin{equation}
\displaystyle \int_{\mathbb{R}^N} (\nabla u \cdot \nabla v +V(\epsilon x)u \cdot v)dx=\displaystyle \int_{\mathbb{R}^N} uv \log u^2dx,\,\, \mbox{for all } v \in C^\infty_0(\mathbb{R}^{N}).
\end{equation}

The main result to be proved is the theorem below.

\begin{theorem}\label{teorema}
Suppose that $V$ satisfies $(V1)$ and ($V2$). Then there is $\epsilon_*>0$ such that problem (\ref{1}) has at least $\l$ positive solutions in $H^{1}(\mathbb{R}^N)$ for all $\epsilon \in (0, \epsilon_*)$.
\end{theorem}

In the proof of Theorem \ref{teorema}, we adapt for our problem some ideas explored in Cao and Noussair \cite{CN}, where the existence and multiplicity of solutions have considered for the following class of problem
\begin{equation}\label{(1.2)}
\left\{
\begin{aligned}
&-\Delta u+u=A(\epsilon x)|u|^{p-2}u~{\rm ~in}
~\mathbb R^N,\\
&u\in H^1(\mathbb R^N).
\end{aligned}
\right.
\end{equation}
 Using Ekeland's variational principle and concentration compactness principle of Lions \cite{lions}, Cao and Noussair proved that if $A$ has $l$ equal maximum points,  then problem  (\ref{(1.2)}) has at least $l$ positive solutions and $l$ nodal solutions if $\epsilon>0$ is small enough.  We would like to point out that different of \cite{CN}, where the energy functional is $C^{1}$, we cannot work directly with the energy functional associated with (\ref{1'}) because it is not continuous, and so, it is not $C^{1}$. Have this in mind, for each $R>0$, we first find a solution $u_{\epsilon,R} \in H_0^{1}(B_R(0))$, and after, taking the limit of $R \to +\infty$ we get a solution for the original problem.

The plan of the paper is as follows. In Section 2 we show some preliminary results which will be used later on. In Section 3 we prove the existence
of multiple solutions for an auxiliary problem, while in Section 4 we prove Theorem 1.1.

\vspace{0.5 cm}

\noindent \textbf{Notation:} From now on in this paper, otherwise mentioned, we use the following notations:
\begin{itemize}
	\item $B_r(u)$ is an open ball centered at $u$ with radius $r>0$, $B_r=B_r(0)$.
	
	\item If $g$ is a mensurable function, the integral $\dis\int_{\mathbb{R}^N}g(x)\,dx$ will be denoted by $\dis\int g(x)\,dx$. Moreover, we denote by $g^{+}$ and $g^{-}$ the positive and negative part of $g$ given by
	$$
	g^{+}(x)=\max\{g(x),0\} \quad \mbox{and} \quad	g^{-}(x)=\max\{-g(x),0\}.
	$$
	
	\item   $C$ denotes any positive constant, whose value is not relevant.
	
	\item  $|\,\,\,|_p$ denotes the usual norm of the Lebesgue space $L^{p}(\mathbb{R}^N)$, for $p \in [1,+\infty]$.
	
	\item  $H_c^{1}(\mathbb{R}^N)=\{u \in H^{1}(\mathbb{R}^N)\,:\, u \,\, \mbox{has compact support}\, \}.$
	
	\item $o_{n}(1)$ denotes a real sequence with $o_{n}(1)\to 0$ as $n \to +\infty$.
	
	\item $2^*=\frac{2N}{N-2}$ if $N \geq 3$ and $2^*=+\infty$ if $N=1,2$.
\end{itemize}

\section{Preliminaries}\label{preliminares}

Hereafter, we consider the problem
\begin{equation}\label{constante}
\left\{
\begin{array}{lc}
-\Delta u+ V(0)u=u \log u^2, & \quad \mbox{in} \quad \mathbb{R}^{N}, \\
u \in H^1(\mathbb{R}^{N}). & \;  \\
\end{array}
\right.
\end{equation}
The corresponding energy functional associated to (\ref{constante}) will be denoted by \linebreak $J_0:H^{1}(\mathbb{R}^N) \rightarrow (-\infty, +\infty]$ and defined as
$$
J_{0}(u)=\dis\frac{1}{2}\int \big (|\nabla u|^2+({V(0)} +1)|u|^2\big)dx-\dis\frac{1}{2}\int u^2\log u^2dx.
$$

In \cite{sz} is proved that problem (\ref{constante}) has a positive solution attained at the infimum
\begin{equation}\label{sisi}
c_0:=\inf_{u \in \mathcal{N}_{0}}J_{0}(u),
\end{equation}
where
$$
\mathcal{N}_{0}=\left\{u \in D(J_{0})\backslash \{0\}: J_{0}(u)=\displaystyle \frac{1}{2} \int |u|^{2}\,dx  \right\}
$$
and
$$
D(J_{0})=\left\{u \in H^{1}(\mathbb{R}^N)\,:\, J_{0}(u)<+\infty \right\}.
$$

\textit{Mutatis mutandis} the previous notations, we shall also use the energy level
$$
c_\infty=\inf_{u \in \mathcal{N}_{\infty}}J_{\infty}(u),
$$
corresponding to problem (\ref{constante}), replacing $V(0)$ by $V_\infty$. Using the definition of $c_0$ and $c_\infty$, it follows that
\begin{equation}\label{cezinho}
c_0 < c_\infty.
\end{equation}

Related to the numbers $c_0$ and $c_\infty$, we would like to point out that they are the mountain pass levels of the functionals $J_0$ and $J_\infty$ respectively.

Following the approach explored in \cite{AlvesdeMorais,cs,sz}, due to the lack of smoothness of $J_{0}$ and $J_\infty$, let us decompose them into a sum of a $C^1$ functional plus a convex lower semicontinuous functional, respectively. For $\delta>0$, let us define the following functions:

$$
F_1(s)=
\left\{ \begin{array}{lc}
0, & \; s=0 \\
-\frac{1}{2}s^2\log s^2, & \; 0<|s|<\delta \\
-\frac{1}{2}s^2(
\log\delta^2+3)+2\delta|s|-\frac{1}{2}\delta^2,  & \; |s| \geq \delta
\end{array} \right.
$$

\noindent and

$$
F_2(s)=
\left\{ \begin{array}{lc}
0, & \;  |s|<\delta \\
\frac{1}{2}s^2\log(s^2/\delta^2)+2\delta|s|-\frac{3}{2}s^2-\frac{1}{2}\delta^2,  & \; |s| \geq  \delta.
\end{array} \right
.$$
Therefore
\begin{equation}\label{efes}
F_2(s)-F_1(s)=\frac{1}{2}s^2\log s^2, \quad \forall s \in \mathbb{R},
\end{equation}
and the functionals $J_0,J_\infty:H^{1}(\mathbb{R}^N) \rightarrow (-\infty, +\infty]$ may be rewritten as

\begin{equation}\label{funcional}
J_0(u)=\Phi_0(u)+\Psi(u) \quad \mbox{and} \quad J_\infty(u)=\Phi_\infty(u)+\Psi(u), \quad u \in H^{1}(\mathbb{R}^N)
\end{equation}
 where
\begin{equation}\label{fi}
\Phi_0(u)=\frac{1}{2}\int (|\nabla u|^{2}+(V(0)+1))|u|^{2})\,dx-\displaystyle \int F_2(u)\,dx,
\end{equation}
\begin{equation}\label{fii}
\Phi_\infty(u)=\frac{1}{2}\int (|\nabla u|^{2}+(V_\infty+1))|u|^{2})\,dx-\displaystyle \int F_2(u)\,dx,
\end{equation}
\noindent and
\begin{equation}\label{psi}
\Psi(u)=\displaystyle \int F_1(u)\,dx.
\end{equation}

It was proved in \cite{cs} and \cite{sz} that $F_1$ and $F_2$ verify the following properties:
\begin{equation}\label{eq1}
F_1, F_2 \in C^1(\mathbb{R},\mathbb{R}).
\end{equation}
If $\delta >0$ is small enough, $F_1$ is convex, even, $F_1(s)\geq 0$ for all $ s\in \mathbb{R}$ and
\begin{equation}\label{eq2}
F'_1(s)s\geq 0, \ s \in \mathbb{R}.
\end{equation}
For each fixed $p \in (2, 2^*)$, there is $C>0$ such that
\begin{equation}\label{eq5}
|F'_2(s)|\leq C|s|^{p-1}, \quad \forall s \in \mathbb{R}.
\end{equation}

\section{An auxiliary functional}\label{tpm}

In what follows, let us fix $R_0>0$ such that $ z_i \in B_{R_0}(0)$ for all $i \in \{0,1,...,\l\}$. Moreover, for each $R>R_0$, we set the functional
 \begin{equation}\label{funcionalA}
J_{\epsilon,R}(u)=\frac{1}{2}\int_{B_R(0)}(|\nabla u|^{2}+(V(\epsilon x)+1)|u|^{2})\,dx-\frac{1}{2}\int_{B_R(0)}u^2\log u^2dx,\ u \in H_0^1(B_R(0)).
\end{equation}
It is easy to check that $J_{\epsilon,R} \in C^{1}(H_0^{1}(B_R(0)),\mathbb{R})$ with
$$
J'_{\epsilon,R}(u)v=\int_{B_R(0)} (\nabla u \nabla v +V(\epsilon x)u v)\,dx -\int_{B_R(0)}v u^2\log u  \,dx, \quad \forall u,v \in H_0^{1}(B_R(0)).
$$
Hereafter, $H_0^{1}(B_R(0))$ is endowed with the norm
\begin{equation} \label{norma}
\left\| u \right\|_\epsilon  = \left( \displaystyle \int (|\nabla u|^2+(V(\epsilon x)+1)|u|^2)dx \right)^{\frac{1}{2}},
\end{equation}
which is also a norm in $H^{1}(\mathbb{R}^N)$. Moreover, this norm is equivalent to the usual norms in $H_0^{1}(B_R(0))$ and $H^{1}(\mathbb{R}^N)$ respectively.

In the sequel,  $\mathcal{N}_{\epsilon,R}$ denotes the Nehari manifold associated with $J_{\epsilon,R}$, that is,
$$
\mathcal{N}_{\epsilon,R}=\left\{u \in H_0^{1}(B_R(0)) \backslash \{0\}: J'_{\epsilon,R}(u)u=0  \right\},
$$
or equivalently,
$$
\mathcal{N}_{\epsilon,R}=\left\{u \in H_0^{1}(B_R(0)) \backslash \{0\}: J_{\epsilon,R}(u)=\displaystyle \frac{1}{2} \int |u|^{2}\,dx  \right\}.
$$

\vspace{0.5 cm}

The next three lemmas show that $J_{\epsilon,R}$ verifies the mountain pass geometry and the well known $(PS)$ condition.

\begin{lemma}\label{vicente} For all $\epsilon>0,R>R_0$, the functional $J_{\epsilon,R}$ has the mountain pass geometry.
\end{lemma}

\begin{proof} \mbox{} \\
\noindent $(i)$: Note that $J_{\epsilon,R}(u)\geq\dis\frac{1}{2}\|u\|_\epsilon^2-\int_{B_R(0)} F_{2}(u)dx$. Hence,  from (\ref{eq5}),  fixed $p \in (2,2^*)$, it follows that
$$
J_{\epsilon,R}(u)\geq \dis\frac{1}{2}\|u\|_\epsilon ^2-C\|u\|_\epsilon^p \geq C_1> 0,
$$
for some $C_1>0$ and $\|u\|_\epsilon>0$ small enough. Here the constant $C_1$ does not depend on $\epsilon$ and $R>0$. \\
\noindent $(ii)$: Let us fix $u \in D(J_{\epsilon, R}) \setminus \{0\}$ with $supp u \subset B_{R_0}(0)$ and $s>0$. Using (\ref{efes}) we get
$$
J_{\epsilon,R}(su)=\dis\frac{s^2}{2}\|u\|_\epsilon^2-\dis\frac{1}{2}\int_{B_{R_0}(0)} s^2u^2\log(|su|^2)dx= s^2 \left[J_{\epsilon, R}(u)-\log s \dis\int_{B_{R_0}(0)} u^2dx\right]\rightarrow - \infty, \, \mbox{as} \, s \rightarrow +\infty.
$$
Thereby, there is $s_0>0$ independent of $\epsilon>0$ small enough and $R>R_0$ such that $J_{\epsilon,R}(s_{0}u)<0$.
\end{proof}

\begin{lemma}\label{limitado}
All (PS) sequences of $J_{\epsilon,R}$ are bounded in $H_0^1(B_R(0))$.
\end{lemma}

\begin{proof} Let $(u_n) \subset H_0^1(B_R(0))$ be a $(PS)_d$ sequence. Then,
$$
\int_{B_R(0)} |u_n|^2dx=2J_{\epsilon,R}(u_n)-J'_{\epsilon,R}(u_n)u_n=2d+o_n(1)+o_n(1)\|u_n\|_{\epsilon} \leq C+o_n(1)\|u_n\|_{\epsilon},
$$
for some $C>0$. Consequently
\begin{equation}\label{rimsky}
|u_n|^2_{L^2(B_R(0))}\leq  C+o_n(1)\|u_n\|_{\epsilon}.
\end{equation}
Now, let us employ the following logarithmic Sobolev inequality found in \cite{lieb},
\begin{equation}\label{xerazade0}
\int u^2 \log u^2dx \leq\frac{a^2}{\pi}|\nabla u|^2_{L^2(\mathbb{R}^{N})}+ \big( \log |u|^2_{L^2(\mathbb{R}^{N})}-N(1+\log a)\big)|u|^2_{L^2(\mathbb{R}^{N})}
\end{equation}
for all $a>0$. Fixing $\frac{a^2}{\pi}=\frac{1}{4}$ and $\xi \in (0,1)$, the inequalities (\ref{rimsky}) and (\ref{xerazade0}) yield

\begin{eqnarray}
 \int_{B_R(0)} u_n^2 \log u_n^2 dx  &\leq&\frac{1}{4}|\nabla u_n|_{L^2(B_R(0))}^2+ C\big(  \log {|u_n|^2_{L^2(B_R(0))}+1} \big)
{|u_n|^2_{L^2(B_R(0))}}\nonumber\\
&\leq&\frac{1}{4}|\nabla u_n|^2_{2}+C_1\big(1+\|u_n\|_{\epsilon}\big)^{1+\xi} \label{isso}.
\end{eqnarray}
Since
$$
d+o_n(1)=J_{\epsilon,R}(u_n)=\frac{1}{2}|\nabla u_n|^2_{2}+  \int_{B_R(0)}(V(\epsilon x)+1)|u_n|^2dx-\frac{1}{2}\int_{B_R(0)} u_n^2 \log u_n^2dx
$$
assertion (\ref{isso}) assures that

\begin{equation}\label{porta}
d+o_n(1) \geq C \big[\| u_n\|^2_{\epsilon}-\big(1+\|u_n\|_{\epsilon}\big)^{1+\xi} \big],
\end{equation}
showing that the sequence $(u_n)$ is bounded in $ H_0^{1}(B_R(0))$.

\end{proof}

\begin{lemma} The mountain pass level $c_{\epsilon,R}$ of $J_{\epsilon,R}$ can be characterized by
$$
c_{\epsilon,R}=\inf_{u \in \mathcal{N}_{\epsilon,R}}J_{\epsilon,R}(u).
$$	
\end{lemma}
\begin{proof} See \cite[Lemma 3.3]{AlvesdeMorais}.
	
\end{proof}

\begin{lemma} \label{compacidadeR} The functional $J_{\epsilon,R}$ satisfies the $(PS)$ condition.
	
\end{lemma}

\begin{proof} Let $(u_n) \subset H_0^1(B_R(0))$ be a $(PS)_d$ sequence for $J_{\epsilon,R}$, that is,
$$
J_{\epsilon,R}(u_n) \to d \quad \mbox{and} \quad J'_{\epsilon,R}(u_n) \to 0.
$$	
By Lemma \ref{limitado}, we can assume that there is $u \in H_0^1(B_R(0))$ and a subsequence of $(u_n)$, still denoted by itself, such that
$$
u_n \rightharpoonup u \quad \mbox{in} \quad H_0^1(B_R(0))
$$
$$
u_n \to u \quad \mbox{in} \quad L^{q}(B_R(0)), \quad \forall q \in [1,2^*)
$$
and
$$
u_n(x) \to u(x) \quad \mbox{a.e. in} \quad B_R(0).
$$
Setting $f(t)=t\log t^2$, $F(t)=\int_{0}^{t}f(s)\,ds=\frac{1}{2}(t^2 \log t^2 -t^2)$ for all $t \in \mathbb{R}$ and fixing $p \in (2,2^*)$, there is $C>0$ such that
$$
|f(t)| \leq C(1+|t|^{p-1}), \quad \forall t \in \mathbb{R}
$$	
and
$$
|F(t)| \leq C(1+|t|^{p}), \quad \forall t \in \mathbb{R}.
$$	
Hence, by the Sobolev embeddings,
$$
\int_{B_R(0)}f(u_n)u_n\,dx \to \int_{B_R(0)}f(u)u\,dx
$$
and
$$
\int_{B_R(0)}f(u_n)v\,dx \to \int_{B_R(0)}f(u)v\,dx, \quad \forall v \in H_0^{1}(B_R(0)).
$$
Now, using the limits $J'_{\epsilon,R}(u_n)u_n=J'_{\epsilon,R}(u_n)u+o_{n}(1)$, it is easy to see that
$$
\|u_n-u\|^{2}_{\epsilon}=\int_{B_R(0)}f(u_n)u_n\,dx-\int_{B_R(0)}f(u_n)u\,dx+o_{n}(1)=o_{n}(1),
$$
showing the lemma. 	
\end{proof}

In the following, let us fix  $\rho_0>0$ satisfying $\overline{B_{\rho_0}(z_i)}\cap\overline{B_{\rho_0}(z_j)}={\O}$ for $i\neq j$ and $i,j\in\{1,...,l\}$, $\bigcup^l_{i=1}B_{\rho_0}(z_i)\subset B_{R_0}(0)$ and $K_{\frac{\rho_0}{2}}=\bigcup^l_{i=1}\overline{B_{\frac{\rho_0}{2}}(z_i)}$. Moreover, we also set the function $Q_{\epsilon}: H^{1}(\mathbb{R}^N) \setminus\{0\} \rightarrow \mathbb{R}^N$ by
$$
Q_{\epsilon}(u)=\frac{\int \chi(\epsilon x)g(\epsilon x) |u|^{2}\,dx}{ \int g(\epsilon x)|u|^{2}\,dx },
$$
where $\chi:\mathbb{R}^N\rightarrow\mathbb{R}^N$ is given by
$$
\chi(x)=\begin{cases}x,~~~~~~~ \text{if}~~ |x|\leq R_0,\\ R_0\frac{x}{|x|},~~\text{if}~~~ |x|>R_0, \end{cases}
$$
and $g:\mathbb{R}^N \to \mathbb{R}$ is a radial positive continuous function such that
\begin{equation} \label{g}
g(z_i)=1, \quad \forall i \in \{1,...,l\} \quad \mbox{and} \quad g(x) \to 0 \quad \mbox{as} \quad |x| \to +\infty.
\end{equation}
The following lemma is very useful to obtain $(PS)_c$ sequences associated with $J_{\epsilon,R}$.
\begin{lemma}\label{Lem:3.2a}
	There exist $\alpha_0>0$, $\epsilon_1 \in (0,\epsilon_0)$ small enough and $R_1>R_0$ large enough such that if $u\in {\mathcal N}_{\epsilon,R}$ and $J_{\epsilon,R}(u) \leq c_0+\alpha_0$, then
	$Q_\epsilon (u)\in K_{\frac{\rho_0}{2}}, \,\, \forall \epsilon  \in (0,\epsilon_1)$ and $ R \geq R_1$.
\end{lemma}
\begin{proof} If the lemma is not true, then exist $\alpha_n\rightarrow 0,~\epsilon_n \to 0, R_n \to +\infty $ and $u_n \in {\mathcal N}_{\epsilon_n,R_n}$ such that
	$$
 J_{\epsilon_n,R_n}(u_n)\leq c_{0}+\alpha_n
	$$
	and
	$$
	Q_{\epsilon_n}(u_n)\not\in K_{\frac{\rho_0}{2}}.
	$$
Since $c_{\epsilon_n,R_n} \geq c_0$, the above inequality gives
$$
c_0\leq c_{\epsilon_n,R_n}\leq J_{\epsilon_n,R_n}(u_n)\leq c_{0}+\alpha_n,
$$
and so,
$$
u_n \in {\mathcal N}_{\epsilon_n,R_n} \quad \mbox{and} \quad J_{\epsilon_n,R_n}(u_n)= c_{\epsilon_n,R_n} +o_n(1).
$$
Setting the functional $\Psi_{\epsilon_n,R_n}:H_0^{1}(B_{R_n}) \to \mathbb{R}$ by
$$
\Psi_{\epsilon_n,R_n}(u)=J_{\epsilon_n,R_n}(u)-\frac{1}{2} \int_{B_{\epsilon_n,R_n(0)}} |u|^{2}\,dx,
$$
we derive
$$
{\mathcal N}_{\epsilon_n,R_n}=\left\{u \in H_0^{1}(B_R(0)) \backslash \{0\}: \Psi_{\epsilon_n,R_n}(u)=0  \right\}.
$$
A simple computation ensures that there is $\beta>0$, which is independent of $n$,  such that
$$
\Psi'_{\epsilon_n,R_n}(u)u=-\int_{B_{\epsilon_n,R_n(0)}} |u|^{2}\leq -\beta, \quad \forall n \in \mathbb{N},
$$
otherwise $c_{\epsilon_n,R_n} \to 0$, which is absurd. With this information in  our hands, we can apply the Ekeland Variational principal found in \cite[Theorem 8.5]{Willem} to assume,  without loss of generality, that  $\|J'_{\epsilon_n,B_{R_n}}(u_n)\| \to 0$ as $n \to +\infty$.

By \cite[Section 6]{AlvesdeMorais}, we need to consider the following two cases:
\begin{itemize}
	\item[\rm (a)]$u_n\rightarrow u \neq 0$ in $L^2(\mathbb{R}^N)$,  \\
	or
	
    \item[\rm (b)] There exists $\{y_n\}\subset\mathbb{R}^N$ such that $\vartheta_n=u_n(.+y_n) \to \vartheta \not=0 $ in $L^2(\mathbb{R}^N)$.
\end{itemize}

\noindent Since $J_{\epsilon_n,R_n}(u_n)=\dis \frac{1}{2}\int_{B_{R_n}(0)} |u_n|^{2} \geq c_0>0$, we derive that $\dis \liminf_{n\rightarrow\infty}t_n >0$. Hence,  the above conclusion ensures that
   \begin{itemize}
	\item[\rm ($a'$)]$u_n\rightarrow u \neq 0$ in $L^2(\mathbb{R}^N)$, for some $u \in H^{1}(\mathbb{R}^N)$ \\
	or
	\item[\rm ($b'$)] There exists $(y_n) \subset\mathbb{R}^N$ such that $v_n=u_n(.+y_n) \to v \not=0 $ in $L^2(\mathbb{R}^N)$, for some $v \in H^{1}(\mathbb{R}^N)$.
\end{itemize}

If $(a')$ holds, we have that
	$$
	Q_{\epsilon_n}(u_n)=\frac{\int \chi(\epsilon_nx)g(\epsilon_n x)|u_n|^{2}dx}{\int g(\epsilon_n x)|u_n|^{2}dx}\rightarrow \frac{\int \chi(0)g(0)|u|^{2}dx}{g(0)\int |u|^{2}dx}=0\in K_{\frac{\rho_0}{2}}.
	$$
	From this, $Q_{\epsilon_n}(u_n)\in K_{\frac{\rho_0}{2}}$ for $n$ large enough, which is a contradiction. \\
Now, if $(b')$ holds, we must distinguish two cases:
\begin{itemize}
	\item[\rm (I)]$|\epsilon_ny_n| \rightarrow+\infty$ \\
	or
	\item[\rm (II)] $\epsilon_ny_n \rightarrow y$ for some $y\in\mathbb{R}^N$, for some subsequence.
\end{itemize}
 If $(I)$  holds, we have that $J'_{\infty}(v)v \leq 0$. Thus, there is $s \in (0,1]$ such that $sv \in {\mathcal N}_{\infty}$
$$
\aligned
2c_{\infty}&\leq 2J_{\infty}(sv)=2J_{\infty}(sv)-J'_{\infty}(sv)sv\\
& =  \int |sv|^{2}\,dx \\
&  \leq \int |v|^{2}\,dx \leq \liminf_{n\rightarrow\infty}\int |v_n|^{2}\,dx= \liminf_{n\rightarrow\infty}\int |u_n|^{2}\,dx\\
&  = \lim_{n \rightarrow +\infty}2J_{\epsilon_n,R_n}(u_n)=2c_0,
\endaligned
$$
which contradicts (\ref{cezinho}).

Now, if (II) holds, the previous argument yields
\begin{equation}\label{(3.10a)}
\aligned
c_{V(y)}\leq c_0,
\endaligned
\end{equation}
where $c_{V(y)}$ is the mountain pass level of the functional $J_{V(y)}:H^{1}(\mathbb{R}^N)\rightarrow\mathbb{R}$ given by
$$
J_{V(y)}(u)=\frac{1}{2}\displaystyle \int (|\nabla u|^2+(V(y)+1)|u|^2)dx-\displaystyle \frac{1}{2}\int  u^2 \log u^2\,dx.
$$
One can see that
$$
c_{V(y)}=\inf_{u\in \mathcal {M}_{V(y)}}J_{V(y)}(u),
$$
where
$$
\mathcal {M}_{V(y)}=\left\{u \in D(J_{V(y)})\backslash \{0\}: J_{V(y)}(u)=\displaystyle \frac{1}{2} \int |u|^{2}\,dx  \right\}.
$$
If $V(y)>1$, as in \cite{AlvesdeMorais}, it is possible to prove that $c_{V(y)}>c_0$, which contradicts (\ref{(3.10a)}). Then $V(y)=1$ and $y=z_i$ for some $i=1,...,l$. Hence
$$
\begin{array}{l}
Q_{\epsilon_n}(u_n)=\dis \frac{\int \chi(\epsilon_nx)g(\epsilon_n x)|u_n|^{2}dx}{\int g(\epsilon_n x)|u_n|^{2}dx}=\frac{\int \chi(\epsilon_n(x+y_n))g(\epsilon_n x+\epsilon_ny_n)|v_n|^{2}dx}{\int g(\epsilon_n x+\epsilon_ny_n)|v_n|^{2}dx } \\
\mbox{}\\
\;\;\;\;\;\;\;\;\;\;\;\;\;\;\rightarrow \dis \frac{\int \chi(z_i)g(z_i)|v|^{2}dx}{\int g(z_i)|v|^{2}dx}=z_i\in K_{\frac{\rho_0}{2}}
\end{array}$$
from where it follows that $Q_{\epsilon_n}(u_n) \in K_{\frac{\rho_0}{2}}$ for $n$ large, which is absurd, because we are assuming that
$Q_{\epsilon_n}(u_n) \not\in K_{\frac{\rho_0}{2}}$. This finishes the proof.  \end{proof}

\vspace{0.5 cm}

Next, we specify the following symbols.
$$\aligned
&\Omega^i_{\epsilon,R}=\{u\in {\mathcal N}_{\epsilon,R}:|Q_\epsilon(u)-z_i|<\rho_0\},\\
&\partial \Omega^i_{\epsilon,R}=\{u\in {\mathcal N}_{\epsilon,R}:|Q_\epsilon(u)-z_i|=\rho_0\},\\
&\alpha^i_{\epsilon,R}=\inf_{u\in\Omega^i_{\epsilon,R}}J_{\epsilon,R}(u),\\
& \tilde\alpha^i_{\epsilon,R}=\inf_{u\in\partial\Omega^i_{\epsilon,R}}J_{\epsilon,R}(u).
\endaligned \eqno$$

\begin{lemma}\label{Lem:3.4}
	Given $ \gamma \in (0,(c_{\infty}-c_0)/2)>0$, there exist $\epsilon_2 \in (0, \epsilon_1)$ small enough such that
	$$
	\alpha^i_{\epsilon,R}<c_0+\gamma~~and~~\alpha^i_{\epsilon,R}<\tilde\alpha^i_{\epsilon,R},
	$$
	for all $\epsilon \in (0,\epsilon_2)$ and $R \geq R_1=R_1(\epsilon)>R_0$. 
\end{lemma}
\begin{proof}Let $u\in H^{1}(\mathbb{R}^N)$  be  a ground state solution for $J_0$, i.e.,
	$$
	u \in \mathcal{N}_0, \quad J_0(u)=c_0 \quad \mbox{and} \quad J'_0(u)=0.
	$$
Hereafter, for each $w \in D(J_0)$, $J'_0(w):H_c^{1}(\mathbb{R}^N) \to \mathbb{R}$ means the functional given by
$$
\langle J'_0(w),z\rangle=\langle \Phi_0'(w), z\rangle-\int F'_1(w)z\,dx, \quad \forall z \in H_c^{1}(\mathbb{R}^N)
$$
and
$$
\|J'_0(w)\|=\sup\left\{\langle  J'_0(w), z\rangle\,:\, z \in H_c^{1}(\mathbb{R}^N) \quad \mbox{and} \quad \|z\|_\epsilon \leq 1 \right\}.
$$
If $\|J'_0(w)\|$ is finite, then $J_0'(w)$ may be extended to a bounded operator in $H^{1}(\mathbb{R}^N)$, and so,  it can be seen as an element of $(H^{1}(\mathbb{R}^N))'$.

For any $1\leq i\leq l$, there is $\epsilon_1>0$ such that
$$
|Q_{\epsilon}(u(\cdot-{z_i}/{\epsilon}))-z_i|< \rho, \quad \forall \epsilon \in (0, \epsilon_1).
$$
Now, we fix $R>R_1=R_1(\epsilon)$ and $t_{\epsilon,R}>0$ such that function $u^i_{\epsilon,R}(x)=t_{\epsilon,R}\varphi_{R}(x) u(x-\frac{z_i}{\epsilon}) \in {\mathcal N}_{\epsilon,R} $,
$$
|Q_{\epsilon}(u^i_{\epsilon,R})-z_i|< \rho, \quad \forall \epsilon \in (0, \epsilon_1) \quad \mbox{and} \quad R>R_1.
$$
and
$$
J_{\epsilon,R}({u}^{i}_{\epsilon,R}) \leq c_0+\frac{\alpha_0}{8}, \quad \forall \epsilon \in (0, \epsilon_1) \quad \mbox{and} \quad R>R_1.
$$
Here, $\varphi_R(x)=\varphi(\frac{x}{R})$ with $\varphi \in C_0^{\infty}(\mathbb{R}^N)$, $0 \leq \varphi(x) \leq 1$ for all $x \in \mathbb{R}^N$, $\varphi(x)=1$ for $x \in B_{1/2}(0)$ and $\varphi(x)=0$ for $x \in B^c_{1}(0)$. Therefore,
$$
u^i_{\epsilon,R} \in \Omega^i_{\epsilon, R}, \quad \forall \epsilon \in (0,\epsilon_2) \quad \mbox{and} \quad R \geq R_1,
$$
from it follows that
\begin{equation} \label{NOVAEQ0}
	\alpha^i_{\epsilon,R}<c_0+\frac{\alpha_0}{4},~~\forall \epsilon \in (0,\epsilon_2) \quad \mbox{and} \quad R \geq R_1.
\end{equation}
	Then, decreasing $\alpha_0$ if necessary,
	$$
	\alpha^i_{\epsilon,R}<c_0+\gamma,~~\forall \epsilon \in (0,\epsilon_2) \quad \mbox{and} \quad R \geq R_1,
	$$
	which is the first inequality. To obtain the second one, note that if $u\in\partial\Omega^i_{\epsilon,R}$, then
	$$
	u\in {\mathcal N}_{\epsilon,R} \quad \mbox{and} \quad |Q_{\epsilon,R}(u)-z_i|=\rho_0>\frac{\rho_0}{2},
	$$
	that is, $Q_{\epsilon,R}(u) \not\in K_{\frac{\rho_0}{2}}$. Thus, from Lemma \ref{Lem:3.2a},
	$$
	J_{\epsilon,R}(u)>c_0+\alpha_0,\,\, \mbox{for all}\,\, u\in\partial\Omega^i_{\epsilon,R}\,\,\,\mbox{and}\,\,\, \forall \epsilon \in (0,\epsilon_2) \quad \mbox{and} \quad R \geq R_1,
	$$
	and so
	\begin{equation} \label{NOVAEQ1}
	\tilde\alpha^i_{\epsilon,R}=\inf_{u\in\partial\Omega^i_{\epsilon,R}}J_{\epsilon,R}(u)\geq c_0 +\alpha_0,~~\forall \epsilon \in (0, \epsilon_2) \quad \mbox{and} \quad R \geq R_1.
	\end{equation}
	Consequently, from (\ref{NOVAEQ0})-(\ref{NOVAEQ1}),
	$$
	\alpha^i_{\epsilon,R}<\tilde\alpha^i_{\epsilon,R}, \quad \mbox{for all} \quad \epsilon \in (0, \epsilon_2) \quad \mbox{and} \quad R \geq R_1,
	$$
	and the results are derived by fixing $\epsilon_2 \in (0,\epsilon_1)$.
\end{proof}

\vspace{0.5 cm}
\begin{theorem} \label{TEOAUX} There are $\epsilon_* \in (0,\epsilon_2)$ small enough and $R_1=R_1(\epsilon)>R_0$ large enough such that $J_{\epsilon,R}$ has at least  $l$ nontrivial critical points for $\epsilon \in (0,\epsilon_*)$ and $R \geq R_1$. Moreover, all of the solutions are positive.
	
\end{theorem}
\begin{proof}From Lemma \ref{Lem:3.4}, there exist $0<\epsilon_* <\epsilon_2$ small enough and $R_1>R_0$ large enough such that
$$
\alpha^i_{\epsilon,R}<\tilde\alpha^i_{\epsilon,R}, \quad \mbox{for all} \quad \epsilon\in(0,\epsilon_*) \quad \mbox{and} \quad R \geq R_1.
$$
Arguing as in \cite[Proof of Theorem 2.1]{CN}, the above inequality permits to use the Ekeland's variational principle to get a $(PS)_{\alpha^i_{\epsilon,R}}$ sequence $(u^i_n)\subset \Omega^i_{\epsilon,R}$ for $J_{\epsilon,R}$. Noting that $\alpha^i_{\epsilon,R}<c_0+\gamma$, from Lemma \ref{compacidadeR} there exists $u^i$ such that $u_n^i\rightarrow u^i$ in $H_0^{1}(B_R(0))$. So
$$
u^i\in\Omega^i_{\epsilon,R},~~~J_{\epsilon,R}(u^i)=\alpha^i_{\epsilon,R}~~and~~J'_{\epsilon,R}(u^i)=0.
$$
Since
$$\aligned
Q_{\epsilon}(u^i)\in \overline{B_{\rho_0}(z_i)},~ Q_{\epsilon}(u^j)\in \overline{B_{\rho_0}(z_j)},\\
\overline{B_{\rho_0}(z_i)}\cap \overline{B_{\rho_0}(z_j)}=\emptyset \quad \mbox{for} \quad i\neq j.
\endaligned $$
We deduce that $u^i\neq u^j$ for $i\neq j$ for $1\leq i, j\leq l$. Hence $J_{\epsilon,R}$ possesses at least  $l$ nontrivial critical points
for all $\epsilon \in (0,\epsilon_*)$ and $ R \geq R_1$. Finally, decreasing {\color{blue}{$\gamma$}} and increasing $R_1$ if necessary, we can assume that
$$
2c_{\epsilon,R}< c_0+\gamma, \quad \mbox{for} \quad \epsilon \in (0, \epsilon^*) \quad \mbox{and} \quad R \geq R_1.
$$
The above inequality permits to conclude that all of the solutions do not change sign, and as $f(t)=t\log t^2$ is an odd function, we can assume that they are nonnegative. Now, the positivity of the solutions in $B_R(0)$ follows by maximum principle. \end{proof}

\section{Existence of solution for original problem}

In the following, for each $i \in \{1,...,\l\}$ and $\epsilon \in (0,\epsilon_*)$, we set $R_n \to +\infty$ and $u_n^{i}=u_{\epsilon,R_n}^{i}$ be a solution obtained in Theorem \ref{TEOAUX}. Then,
$$
\int_{B_{R_n}}(\nabla u^i_n \nabla v+V(\epsilon x)u^i_n v)\,dx=\int_{B_{R_n}}u^i_n \log \vert u^i_n\vert^{2}v \,dx, \quad \forall v \in H^{1}(B_{R_n}(0))
$$
and
$$
J_{\epsilon,R_n}(u^i_n)=\alpha^i_{\epsilon,R_n}, \quad \forall n \in \mathbb{N}.
$$

\begin{proposition}\label{lema1}
	There exists $u^{i} \in H^{1}(\mathbb{R}^N)$ such that $u^{i}_n\rightharpoonup u^{i}$ in $H^{1}(\mathbb{R}^N)$ and $u^{i} \not=0$ for all $i \in \{1,..,l\}$.
\end{proposition}

\begin{proof} Since $(\alpha^i_{\epsilon,n})$ is a bounded sequence, it is easy to check that $(u^i_n)$ is a bounded sequence. Hence, we may assume that $u^i_n\rightharpoonup u^i$ for some $u^i \in H^{1}(\mathbb{R}^N)$. Arguing by contradiction, we assume that there is $i_0 \in \{1,..l\}$ such that $u^{i_0}=0$. In the sequel $(u_n)$ and $(\alpha_n)$ denote $(u^i_n)$ and $(\alpha^i_{\epsilon,R_n})$ respectively.
	
To proceed further we need to use the Concentration Compactness Principle, due to Lions \cite{lions}, employed to the following sequence
	$$
	\rho_n(x):= \frac{|u_n(x)|^2}{|u_n|^{2}_{2}}, \quad \forall x \in \mathbb{R}^N.
	$$
	This principle assures that one and only one of the following statements holds for a subsequence of $(\rho_n)$, still denoted by itself:
	
	\begin{description}
		\item [{\it (Vanishing)}]
		
		\begin{equation}\label{vanishing}
		\displaystyle \lim_{n \to +\infty}\sup_{y \in \mathbb{R}^N} \int_{B_K(y)}\rho_n dx = 0, \quad \forall K >0;
		\end{equation}
		
		\item [{\it (Compactness)}] There exists a sequence of points $(y_n) \subset \mathbb{R}^N$ such that for all $\eta > 0$, there exists $K > 0$ such that
		
		\begin{equation}
		\int_{B_K(y_n)}\rho_n dx \geq 1 - \eta, \quad \forall n \in \mathbb{N};
		\label{compactness}
		\end{equation}

		\item [{\it (Dichotomy)}] There exist $(y_n) \subset \mathbb{R}^N$, $\alpha \in (0,1)$, $K_1 > 0$, $K_n \to +\infty$ such that the functions $\displaystyle \rho_{1,n}(x) := \chi_{B_{K_1}(y_n)}(x)\rho_n(x)$ and $\displaystyle \rho_{2,n}(x) := \chi_{B_{K_n}^c(y_n)}(x)\rho_n(x)$ satisfy
		
		\begin{equation}\
		\int \rho_{1,n} dx \to \alpha \quad \mbox{and} \quad \int \rho_{2,n} dx \to 1 - \alpha.
		\label{dichotomy}
		\end{equation}
	\end{description}
	
	Our objective is to show that $(\rho_n)$ verifies the {\it Compactness} condition and in order to do so we act by excluding the others two possibilities. But this fact will lead to a contradiction, showing the proposition.
	
	The vanishing case (\ref{vanishing}) can not occur, otherwise we conclude that $|u_n|_{p} \rightarrow 0$, and so, $F'_2(u_n)u_n \to 0$ in $L^{1}(\mathbb{R}^N)$. Arguing as in the previous section,  it is possible to prove that $u_n \to 0$ in $H^{1}(\mathbb{R}^N)$. However, this convergence contradicts the fact that $\alpha_n \geq C_1$ for all $n \in \mathbb{N}$, see Lemma \ref{vicente}.

	Let us show that {\it Dichotomy} also does not hold. Suppose that this is not the case. Under  this assumption, we claim that $(y_n)$ is unbounded, because otherwise, in this case, using the fact that $|u_n|_{L^{2}(\mathbb{R}^N)} \not\to 0$, the first convergence in (\ref{dichotomy}) leads to
	$$
	\int_{B_{K_1}(y_n)}|u_n|^2dx =|u_n|_{2}^2 \int_{\mathbb{R}^N}\rho_{1,n} dx \geq \,
	\delta
	$$
	for some $\delta>0$ and for sufficiently large $n$. Then, picking $R' > 0$ such that $B_{K_1}(y_n) \subset B_{R'}(0)$, for all $n \in \mathbb{N}$, it follows that
	$$
	\int_{B_{R'}(0)}|u_n|^2dx \geq \delta, \quad \mbox{for all $n$ sufficiently large.}
	$$
	Since $u_n \to 0$ in $L^{2}(B_{R'}(0))$,  the above inequality is impossible. Thereby $(y_n)$ is an unbounded sequence. In what follows, we set
	\begin{eqnarray}\label{sequencia}
	v_n(x):=u_n(x+y_n), \ x \in \mathbb{R}^{N}.
	\end{eqnarray}
	Hence $(v_n) \subset H^1(\mathbb{R}^{N})$ is bounded and, up to subsequence, we may assume that $v_n\rightharpoonup v$ and by the first part of (\ref{dichotomy}) we have $v \not \equiv 0$.
	
	\begin{claim} \label{A2} $F'_1(v)v \in L^{1}(\mathbb{R}^N)$ and $J_\infty'(v)v \leq 0.$	
		
	\end{claim}

	Note that, if $\varphi \in C^\infty_0(\mathbb{R}^N)$, $0 \leq \varphi \leq 1$, $\varphi \equiv 1$ in $B_1(0)$ and $\varphi \equiv 0$ in $B_{2}(0)^c$, defining  $\varphi_R := \varphi (\cdot/R)$ and $v=\varphi_ R(\cdot -y_n) u_n$, the following equality holds
	$$
	\int \big (\nabla v_n \cdot \nabla (\varphi_R v_n)+({V(\epsilon (x+y_n))} +1)\varphi_R (v_n)^2\big)dx + \int F'_1(v_n)v_n \varphi_Rdx
	$$
	\begin{equation}\label{calma2}
	= \int F'_2(v_n)v_n \varphi_Rdx+o_n(1).
	\end{equation}
	Fixing $R$ and passing to the limit in the above equality when $n \rightarrow \infty$ we get
	$$
	\int \big (\varphi_R |\nabla v|^2+ v \nabla \varphi_R \cdot \nabla v)+(V_\infty +1)\varphi_R v^2\big)dx + \int F'_1(v)v \varphi_R dx\leq \int F'_2(v)v \varphi_Rdx.
	$$
	Now, the claim follows, using that $F'_1(t)t \geq 0$ for all $t \in \mathbb{R}$, and applying Fatou's lemma in the last inequality, as $R \rightarrow + \infty$.

	Therefore, there is  $t_\infty \in (0,1]$ such that $t_\infty v \in {\mathcal N}_\infty$, and so,
	
	$$
	\begin{array}{l}
	c_\infty \leq J_\infty(t_\infty v)=\dis \frac{t_\infty^2}{2} \int|v
	|^2 dx
	\leq \liminf_{n \rightarrow \infty}  \frac{1}{2} \int|v_n|^2dx\leq \limsup_{n \rightarrow \infty}  \frac{1}{2} \int|u_n|^2dx \\
	\;\;\;\;\;\; = \dis \limsup_{n \rightarrow \infty}  J_{\epsilon_n,R_n}(u_n)= \limsup_{n \rightarrow \infty}\alpha_n \leq c_0+\gamma,
	\end{array}
	$$
	which contradicts the fact that $\gamma < c_\infty -c_0$.  Thus, in any case,  \textit{Dichotomy} does not occur and, actually, \textit{Compactness} must hold. To reach our goal let us state the last claim.
	
	\vspace{.2cm}
	
	\begin{claim} \label{A3}
		The sequence of points	$(y_n) \subset \mathbb{R}^N$ in (\ref{compactness}) is bounded.
	\end{claim}
	The proof of this claim consists in assuming by contradiction that the sequence of points $(y_n)$ is unbounded. Then, up to subsequence, $|y_n| \to +\infty$, and we proceed as in the case of {\it Dichotomy}, where $(y_n)$ was unbounded, reaching that $c_0+\gamma \geq c_\infty$.
	
	In view of Claim \ref{A3}, for a given $\eta > 0$, there exists $R > 0$ such that, by (\ref{compactness}),
	$$
	\int_{B_R^c(0)}\rho_n dx < \eta, \quad \forall n \in \mathbb{N},
	$$
or equivalent to
	\begin{equation}
	\int_{B_R^c(0)}|u_n|^2 dx \leq \eta|u_n|^2_{2} \leq b\eta, \quad \forall n \in \mathbb{N},
	\label{decay1}
	\end{equation}
	where $b=\displaystyle \sup_{n \in \mathbb{N}}|u_n|^{2}_{2}$. Then, for $R_1 \geq \max\{R,R'_0\}$, due to the convergence $u_n \to 0$ in $L^2(B_{R_1}(0))$, there exists $n_0 \in \mathbb{N}$ such that
	\begin{equation}
	\int_{B_{R_1}(0)}|u_n|^2 dx \leq \eta, \quad \forall n \geq n_0.
	\label{decay3}
	\end{equation}
	Then, by (\ref{decay1}) and (\ref{decay3}), it follows that if $n \geq n_0$,
	$$
	 \int |u_n|^2 dx \leq \eta + \int_{B_{R_1}^c(0)}|u_n|^2 dx \leq \eta + b\eta  \leq C\eta
	$$
	for some $C$ that does not depend on $\eta$. As $\eta$ is arbitrary, we can conclude that $u_n \to 0$ in $L^2(\mathbb{R}^N)$. Since $(u_n)$ is bounded in $H^{1}(\mathbb{R}^N)$, by interpolation on the Lebesgue spaces, it follows that
	$$
	u_n \to 0 \quad \mbox{in $L^p(\mathbb{R}^N)$, for all $2 \leq p < 2^*$}.
	$$
However, this limit implies that $J_{\epsilon,R_n}(u_n)=\alpha_{n}\to0$, which is impossible, because $\alpha_n \geq c_\epsilon>0$ for all $n \in \mathbb{N}$. \end{proof}

As an immediate consequence of Proposition \ref{lema1}, we have the corollary.
\begin{corollary}\label{corolario2}
	For each sequence $(u^i_n) \subset H^{1}(\mathbb{R}^N)$ given in Proposition \ref{lema1} and for small $\epsilon \in (0,\epsilon_*)$, we have that $u^i \not=0$ and $J'_\epsilon(u^i)v=0$ for all $v \in C_{0}^{\infty}(\mathbb{R}^N)$. Moreover, the following limits hold 	
\begin{equation} \label{Q1}
Q_{\epsilon}(u^i_n) \to Q_{\epsilon}(u^i),\quad i=1,2,\cdots, l.
\end{equation}
Since
$$
Q_{\epsilon}(u^i_n)\in \overline{B_{\rho_0}(z_i)}, \quad \forall n \in \mathbb{N},
$$
we have that
\begin{equation} \label{Q2}
Q_{\epsilon}(u^i)\in \overline{B_{\rho_0}(z_i)}.
\end{equation}
\end{corollary}

\begin{proof} By Proposition \ref{lema1}, we know that $u^i\not=0$ for all $i \in \{1,...,l\}$. The limit $u^i_n \to u^i$ in $L_{loc}^{q}(\mathbb{R}^N)$ for all $q \in [2,2^*)$  ensures that
$$
\int u^i_n \log |u^i_n|^2 v \,dx \rightarrow \int u^i \log |u^i|^2 v  \, dx, \quad \forall v \in C_0^{\infty}(\mathbb{R}^N).
$$
Since
$$
\int (\nabla u^i_n \nabla v + (V(\epsilon x)+1)u^i_nv)\,dx \to \int (\nabla u^i \nabla v + (V(\epsilon x)+1)u^iv)\,dx,  \quad \forall v \in C_0^{\infty}(\mathbb{R}^N),
$$
we can conclude that $J'_\epsilon(u^i)v=0$ for all $v \in C_{0}^{\infty}(\mathbb{R}^N)$. Using the fact that $g(x) \to 0$ as $|x| \to +\infty$, it is easy to show that
$$
\int \chi(\epsilon x)g(\epsilon x) |u^i_n|^{2}\,dx \to \int \chi(\epsilon x)g(\epsilon x) |u^i|^{2}\,dx
$$
and
$$
\int g(\epsilon x)|u^i_n|^{2}\,dx \to \int g(\epsilon x)|u^i|^{2}\,dx.
$$
The above limits ensure that (\ref{Q1}) and (\ref{Q2}) hold.  \end{proof}

\subsection{Proof of Theorem \ref{teorema}}

By Corollary \ref{corolario2}, for each $i \in \{1,...,\l\}$ and $\epsilon \in (0,\epsilon_*)$, there is a solution $u^i \in H^{1}(\mathbb{R}^N)\backslash\{0\}$ for problem (\ref{1'}) such that
$$
Q_{\epsilon}(u^i)\in \overline{B_{\rho_0}(z_i)}.
$$
Since
$$
\overline{B_{\rho_0}(z_i)} \cap \overline{B_{\rho_0}(z_j)}= \emptyset \quad \mbox{and} \quad i \not= j,
$$
it follows that $u^i \not= u^j$ for $i \not= j$.  Due to a change of variable,
the functions \linebreak $v^i(x)=u^i(x/\epsilon), \forall x\in \mathbb{R}^N$, $i \in \{1,...,\l\}$ are $l$ positive solutions of problem (\ref{1}).

\newpage

\noindent \textsc{Claudianor O. Alves } \\
Unidade Acad\^{e}mica de Matem\'atica\\
Universidade Federal de Campina Grande \\
Campina Grande, PB, CEP:58429-900, Brazil \\
\texttt{coalves@mat.ufcg.edu.br} \\
\noindent and \\
\noindent \textsc{Chao Ji}(Corresponding Author) \\
Department of Mathematics\\
East China University of Science and Technology \\
Shanghai 200237, PR China \\
\texttt{jichao@ecust.edu.cn}

\end{document}